\documentclass[twoside,11pt]{article}

\usepackage{graphicx,subcaption}
\usepackage[top=1in,bottom=1in,left=1in,right=1in]{geometry}
\usepackage[numbers,sort&compress]{natbib} 
\usepackage{amsfonts, amsmath, amssymb, amsthm, bbm}
\usepackage[normalem]{ulem}
\usepackage{comment}

\newcounter{alphasect}
\def\alphainsection{0}

\let\oldsection=\section
\def\section{%
	\ifnum\alphainsection=1%
	\addtocounter{alphasect}{1}
	\fi%
	\oldsection}%

\renewcommand\thesection{%
	\ifnum\alphainsection=1%
	\Alph{alphasect}%
	\else
	\arabic{section}%
	\fi%
}%

\newenvironment{alphasection}{%
	\ifnum\alphainsection=1%
	\errhelp={Let other blocks end at the beginning of the next block.}
	\errmessage{Nested Alpha section not allowed}
	\fi%
	\setcounter{alphasect}{0}
	\def\alphainsection{1}
}{%
\setcounter{alphasect}{0}
\def\alphainsection{0}
}%

\usepackage[svgnames]{xcolor}
\usepackage{hyperref}
\hypersetup{colorlinks=true, linkcolor=DarkRed, citecolor=DarkGreen, urlcolor=DarkBlue}
\usepackage{url}

\newtheorem{theorem}{Theorem}
\newtheorem{proposition}{Proposition}
\newtheorem{lemma}{Lemma}

\newtheorem{definition}{Definition}
\newtheorem{hyp}{Hypothesis}

\newcommand{\lemref}[1]{Lemma~\ref{lem:#1}}

\DeclareMathOperator*{\argmax}{arg\, max}
\DeclareMathOperator*{\argmin}{arg\, min}

\DeclareMathOperator{\rank}{rank}
\DeclareMathOperator{\tr}{tr}
\DeclareMathOperator{\diag}{diag}
\newcommand{\Diag}[1]{\diag\left(#1\right)}

\usepackage{chngcntr}

\counterwithin*{equation}{section}

\def\1{{\mathbf 1}}

\newcommand{\E}{\operatorname{\mathbb{E}}}

\newcommand{\Var}{\operatorname{Var}}
\newcommand{\Cov}{\operatorname{Cov}}

\newcommand{\expect}[1]{\mathbb{E}\left[#1\right]}
\newcommand{\proba}[1]{\mathbb{P}\left[#1\right]}

\def\bE{\mathbf{E}}

\def\bmu{\boldsymbol{\mu}}

\title{Adaptive Clustering through Semidefinite Programming}
\usepackage{authblk}
\author[]{Martin Royer}
\affil[]{Laboratoire de Math\'ematiques d'Orsay, Univ. Paris-Sud, CNRS, \\Universit\'e Paris-Saclay, 91405 Orsay, France.}

\begin{document}

\date{}
\maketitle
\begin{abstract}
	We analyze the clustering problem through a flexible probabilistic model that aims to identify an optimal partition on the sample $X_1,...,X_n$. We perform exact clustering with high probability using a convex semidefinite estimator that interprets as a corrected, relaxed version of $K$-means. The estimator is analyzed through a non-asymptotic framework and showed to be optimal or near-optimal in recovering the partition. Furthermore, its performances are shown to be adaptive to the problem's effective dimension, as well as to $K$ the unknown number of groups in this partition. We illustrate the method's performances in comparison to other classical clustering algorithms with numerical experiments on simulated data.
\end{abstract}

\section{Introduction}

Clustering, a form of unsupervised learning, is the classical problem of assembling $n$ observations $X_1,...,X_n$ from a $p$-dimensional space into $K$ groups. Applied fields are craving for robust clustering techniques, such as computational biology with genome classification, data mining or image segmentation from computer vision. But the clustering problem has proven notoriously hard when the embedding dimension is large compared to the number of observations (see for instance the recent discussions from \cite{aziz, ery}).

\noindent
A famous early approach to clustering is to solve for the geometrical estimator K-means \cite{steinhaus, lloyd, kmeans}. The intuition behind its objective is that groups are to be determined in a way to minimize the total intra-group variance. It can be interpreted as an attempt to "best" represent the observations by $K$ points, a form of vector quantization. Although the method shows great performances when observations are homoscedastic, K-means is a NP-hard, ad-hoc method. Clustering with probabilistic frameworks are usually based on maximum likelihood approaches paired with a variant of the EM algorithm for model estimation, see for instance the works of Fraley \& Raftery \cite{raftery} and Dasgupta \& Schulman \cite{dasgupta}. These methods are widespread and popular, but they tend to be very sensitive to initialization and model misspecifications.

\noindent
Several recent developments establish a link between clustering and semidefinite programming. Peng \& Wei \cite{pengwei} show that the K-means objective can be relaxed into a convex, semidefinite program, leading Mixon \textit{et al.} \cite{mixon} to use this relaxation under a subgaussian mixture model to estimate the cluster centers. Chr\'etien \textit{et al.} \cite{chretien} use a slightly different form of a semidefinite program, inspired by work on community detection by Gu\'edon \& Vershynin \cite{guedon}, to recover the adjacency matrix of the cluster graph with high probability. Lastly in the different context of variable clustering, Bunea \textit{et al.} \cite{pecok} present a semidefinite program with a correction step to produce non-asymptotic exact recovery results.

\noindent
In this work, we introduce a semidefinite, penalized estimator for point clustering inspired by \cite{pengwei} and adapted from the work and context of \cite{pecok}. We analyze it through a flexible probabilistic model inducing an optimal partition that we aim to recover. We investigate the optimal conditions of exact clustering recovery with high probability and show optimal performances \--- including in high dimensions, improving on \cite{mixon}, as well as adaptability to the effective dimension of the problem. We also show that our results continue to hold without knowledge of the number of groups $K$. Lastly we provide evidence of our method's efficiency from simulated data and suggest a coherent alternative in case our estimator is too costly to compute.

\noindent
\textbf{Notation.} Throughout this work we use the convention $0/0 := 0$ and $[n] = \{1,...,n\}$. We take $a_n \lesssim b_n$ to mean that $a_n$ is smaller than $b_n$ up to an absolute constant factor. Let $\mathcal{S}_{d-1}$ denote the unit sphere in $\mathbb{R}^d$. For $q\in \mathbb{N}^* \cup \{+\infty\}$, $\nu \in \mathbb{R}^d$, $|\nu|_q$ is the $l_q$-norm and for $M \in \mathbb{R}^{d\times d'}$, $|M|_q$, $|M|_F$, $|M|_*$ and $|M|_{op}$ are respectively the entry-wise $l_q$-norm, the Frobenius norm associated with scalar product $\langle ., . \rangle$, the nuclear norm and the operator norm. $|D|_V$ is the variation semi-norm for a diagonal matrix $D$, the difference between its maximum and minimum element. Let $A \succcurlyeq B$ mean that $A-B$ is symmetric, positive semidefinite.

\section{Probabilistic modeling of point clustering}

Consider $X_1,...,X_n$ and let $\nu_a = \expect{X_a}$. The variable $X_a$ can be decomposed into
\begin{align}
\label{latent}
X_a = \nu_a + E_a, \quad a=1,...,n,
\end{align}
with $E_a$ stochastic centered variables in $\mathbb{R}^p$.

\begin{definition}
	For $K>1$, $\bmu = (\mu_1,...,\mu_K) \in (\mathbb{R}^{p})^K$, $\delta \geqslant 0$ and $\mathcal{G} = \{G_1, ... , G_K\}$ a partition of $[n]$,  we say $X_1,...,X_n$ are $(\mathcal{G},\bmu,\delta)$-clustered if $\forall k \in [K],\forall a \in G_k, |\nu_a-\mu_k|_2 \leqslant \delta$. We then call
	\begin{align}
		\Delta(\bmu) := \min_{k<l} | \mu_k - \mu_l |_2
	\end{align} the separation between the cluster means, and
	\begin{align}
		\rho(\mathcal{G},\bmu,\delta) := \Delta(\bmu)/\delta
	\end{align} the discriminating capacity of $(\mathcal{G},\bmu,\delta)$.
\end{definition}

\noindent
In this work we assume that $X_1,...,X_n$ are $(\mathcal{G},\bmu,\delta)$-clustered. Notice that this definition does not impose any constraint on the data: for any given $\mathcal{G}$, there exists a choice of $\bmu$, means and radius $\delta$ important enough so that $X_1,...,X_n$ are $(\mathcal{G},\bmu,\delta)$-clustered. But we are interested in partitions with greater discriminating capacity, i.e. that make more sense in terms of group separation. Indeed remark that if $\rho(\mathcal{G},\bmu,\delta) < 2$, the population clusters $\{\nu_a\}_{a\in G_1}, ..., \{\nu_a\}_{a\in G_K}$ are not linearly separable, but a high $\rho(\mathcal{G},\bmu,\delta)$ implies that they are well-separated from each other. Furthermore, we have the following result.
\begin{proposition}
	\label{prop:identifiability}
	Let $(\mathcal{G}^*_K,\bmu^*, \delta^*) \in \argmax \rho(\mathcal{G},\bmu,\delta)$ for $(\mathcal{G},\bmu,\delta)$ such that $X_1,...,X_n$ are $(\mathcal{G},\bmu,\delta)$-clustered, and $|\mathcal{G}|=K$. If $\rho(\mathcal{G}^*_K,\bmu^*,\delta^*) > 4$ then $\mathcal{G}^*_K$ is the unique maximizer of $\rho(\mathcal{G},\bmu,\delta)$.
\end{proposition}

\noindent
So $\mathcal{G}^*_K$ is the partition maximizing the discriminating capacity over partitions of size $K$. Therefore in this work, we will assume that there is a $K>1$ such that $X_1,...,X_n$ is $(\mathcal{G},\bmu,\delta)$-clustered with $|\mathcal{G}|=K$ and $\rho(\mathcal{G},\bmu,\delta) > 4$. By Proposition \ref{prop:identifiability}, $\mathcal{G}$ is then identifiable. It is the partition we aim to recover.

\noindent
We also assume that $X_1,...,X_n$ are independent observations with subgaussian behavior. Instead of the classical isotropic definition of a subgaussian random vector (see for example \cite{vershynin}), we use a more flexible definition that can account for anisotropy.
\begin{definition}
	Let $Y$ be a random vector in $\mathbb{R}^d$, $Y$ has a subgaussian distribution if there exist $\Sigma \in \mathbb{R}^{d\times d}$ such that $\forall x \in \mathbb{R}^d$,
	\begin{align}
		\expect{e^{x^T (Y-\E Y)} } \leqslant e^{x^T \Sigma x / 2}.
	\end{align}
\end{definition}

\noindent
We then call $\Sigma$ a variance-bounding matrix of random vector $Y$, and write shorthand $Y \sim $ subg$(\Sigma)$. Note that $Y \sim $ subg$(\Sigma)$ implies $\Cov(Y) \preccurlyeq \Sigma$ in the semidefinite sense of the inequality. To sum-up our modeling assumptions in this work:
\begin{hyp}
	\label{hyp:model}
	Let $X_1,...,X_n$ be independent, subgaussian, $(\mathcal{G},\bmu, \delta)$-clustered with $\rho(\mathcal{G},\bmu,\delta) > 4$.
\end{hyp}

\noindent
Remark that the modelization of Hypothesis \ref{hyp:model} can be connected to another popular probabilistic model: if we further ask that $X_1,...,X_n$ are identically-distributed within a group (and hence $\delta = 0$), the model becomes a realization of a \textit{mixture model}.

\section{Exact partition recovery with high probability}

Let $\mathcal{G} = \{G_1,...,G_K\}$ and $m := \min_{k \in [K]} |G_k|$ denote the minimum cluster size. $\mathcal{G}$ can be represented by its caracteristic matrix $B^* \in \mathbb{R}^{n\times n}$ defined as $\forall k,l \in [K]^2, \forall (a,b) \in G_k\times G_l$,
\[
B^*_{ab} := \left\{
\begin{array}{ll}
{1}/{|G_k|} \qquad &\text{if } k = l\\
0 \qquad &\text{otherwise.}
\end{array}
\right.
\]
In what follows, we will demonstrate the recovery of $\mathcal{G}$ through recovering its caracteristic matrix $B^*$. We introduce the sets of square matrices
\begin{align}
	&\mathcal{C}_K^{\{0,1\}} := \{ B \in \mathbb{R}_{+}^{n\times n} : B^T = B, \tr(B) = K, B 1_n = 1_n, B^2 = B \}\\
	&\mathcal{C}_K := \{ B \in \mathbb{R}_{+}^{n\times n} :  B^T = B, \tr(B) = K, B 1_n = 1_n, B \succcurlyeq 0\}\\
	&\mathcal{C} := \bigcup_{K \in \mathbb{N}} \mathcal{C}_K.
\end{align}
We have: $\mathcal{C}_K^{\{0,1\}} \subset \mathcal{C}_K \subset \mathcal{C}$ and $\mathcal{C}_K$ is convex. Notice that $B^* \in \mathcal{C}_K^{\{0,1\}}$. A result by Peng, Wei (2007) \cite{pengwei} shows that the K-means estimator $\bar{B}$ can be expressed as
\begin{align}
	\label{km}
	\bar{B} = \argmax_{B \in \mathcal{C}_K^{\{0,1\}}} \langle {\widehat{\Lambda}},B \rangle
\end{align}
for $\widehat{\Lambda} := (\langle X_a,X_b \rangle)_{(a,b) \in [n]^2} \in \mathbb{R}^{n\times n}$, the observed Gram matrix. Therefore a natural relaxation is to consider the following estimator:
\begin{align}
	\label{sdp}
	\widehat{B} := \argmax_{B \in \mathcal{C}_K} \langle \widehat{\Lambda},B \rangle.
\end{align}
Notice that $\E{\widehat{\Lambda}} = \Lambda + \Gamma$ for $\Lambda := (\langle \nu_a, \nu_b \rangle)_{(a,b) \in [n]^2} \in \mathbb{R}^{n\times n}$, and $\Gamma := \expect{\langle E_a, E_b \rangle}_{(a,b) \in [n]^2} = \Diag{| \Var(E_a) |_*}_{1 \leqslant a \leqslant n} \in \mathbb{R}^{n \times n}$. The following two results demonstrate that $\Lambda$ is the signal structure that lead the optimizations of \eqref{km} and \eqref{sdp} to recover $B^*$, whereas $\Gamma$ is a bias term that can hurt the process of recovery.

\begin{proposition}
	\label{prop:motivation}
	There exist $c_0 > 1$ absolute constant such that if $\rho^2(\mathcal{G},\bmu,\delta) > c_0 (6 + \sqrt{n}/m)$ and $m\Delta^2(\bmu) > 8|\Gamma|_V$, then we have
	\begin{align}
		\argmax_{B \in \mathcal{C}_K^{\{0,1\}}} \langle {\Lambda + \Gamma},B \rangle = B^* = \argmax_{B \in \mathcal{C}_K} \langle {\Lambda + \Gamma},B \rangle.
	\end{align}
\end{proposition}

\noindent
This proposition shows that the $\widehat{B}$ estimator, as well as the K-means estimator, would recover partition $\mathcal{G}$ on the population Gram matrix if the variation semi-norm of $\Gamma$ were sufficiently small compared to the cluster separation. Notice that to recover the partition on the population version, we require the discriminating capacity to grow as fast as $1 + (\sqrt{n}/m)^{1/2}$ instead of simply 1 from Hypothesis \ref{hyp:model}. The following proposition demonstrates that if the condition on the variation semi-norm of $\Gamma$ is not met, $\mathcal{G}$ may not even be recovered on the population version.
\begin{proposition}
	\label{prop:counter}
	There exist $\mathcal{G}, \bmu, \delta$ and $\Gamma$ such that $\rho^2(\mathcal{G},\bmu,\delta) = + \infty$ but we have $m\Delta^2(\bmu) < 2|\Gamma|_V$ and
	\begin{align}
		B^* \notin \argmax_{B \in \mathcal{C}_K^{\{0,1\}}} \langle {\Lambda + \Gamma},B \rangle \quad\text{ and }\quad B^* \notin {\argmax_{B \in \mathcal{C}_K} \langle {\Lambda + \Gamma},B \rangle}.
	\end{align}
\end{proposition}
\noindent
So Proposition \ref{prop:counter} shows that even if the population clusters are perfectly discriminated, there is a configuration for the variances of the noise that makes it impossible to recover the right clustering by K-means. This shows that K-means may fail when the random variable homoscedasticity assumption is violated, and that it is important to correct for $\Gamma$.

\noindent
The estimator from \cite{pecok} can be adapted to our context. We introduce the following estimator, for $(a,b)\in [n]^2$ let $V(a,b) := \max_{(c,d) \in ([n]\setminus\{a,b\})^2}\big|\langle X_{a}-X_{b}, \frac{X_{c}-X_{d}}{|X_{c}-X_{d}|_2}\rangle\big|$, $b_1 := \argmin_{b \in [n]\setminus\{a\}}V(a,b)$ and $b_2:= \argmin_{b\in [p]\setminus\{a,b_1\}}V(a,b)$. Then for $a\in [n]$, let
\begin{align}
	\widehat{\Gamma}^{corr} := \Diag{\langle X_{a}-X_{b_{1}}, X_{a}-X_{b_{2}}\rangle_{a\in [n]}}.
\end{align}

\noindent
Computing $\widehat{\Gamma}^{corr}$ can be interpreted as a correcting term to de-bias $\widehat{\Lambda}$ as an estimator of $\Lambda$. The result from Proposition \ref{prop:motivation} demonstrates the interest of studying the following semi-definite estimator of the projection matrix $B^*$, let
\begin{align}
	\label{coco}
	\widehat{B}^{corr} := \argmax_{B \in \mathcal{C}_K} \langle \widehat{\Lambda} - \widehat{\Gamma}^{corr},B \rangle.
\end{align}

\noindent
In order to demonstrate the recovery of $B^*$ by this estimator, we introduce different quantitative measures of the "spread" of our stochastic variables, that affect the quality of the recovery. By Hypothesis \ref{hyp:model} there exist $\Sigma_1,...,\Sigma_n$ such that $\forall a \in [n]$, $X_a \sim$ subg$(\Sigma_a)$. Let
\begin{align}
\sigma^2 := \max_{a\in[n]} |\Sigma_{a}|_{op}, \quad
\mathcal{V}^2 := \max_{a\in[n]} |\Sigma_{a}|_{F}, \quad
\gamma^2 := \max_{a\in[n]} |\Sigma_{a}|_{*}.
\end{align}
We are now ready to introduce this paper's main result: a condition on the separation between the cluster means sufficient for ensuring recovery of $B^*$ with high probability.
\begin{theorem}
	\label{thm:mainthm}
	Assume that $m>2$. For $c_1,c_2>0$ absolute constants, if
	\begin{align}
	\label{sep4}
	m\Delta^2(\bmu) &\geqslant c_2 \big( \sigma^2 ( n + m\log n ) + \mathcal{V}^2 ( \sqrt{n + m\log n} ) + \gamma(\sigma\sqrt{\log n}+\delta) + \delta^2 (\sqrt{n}+m) \big),
	\end{align}
	then with probability larger than $1-c_1/n$ we have $\widehat{B}^{corr} = B^*$, and therefore $\widehat{\mathcal{G}}^{corr} = \mathcal{G}$.
\end{theorem}

\noindent
We call the right hand side of \eqref{sep4} the separating rate. Notice that we can read two kinds of requirements coming from the separating rate: requirements on the radius $\delta$, and requirements on $\sigma^2,\mathcal{V}^2,\gamma$ dependent on the distributions of observations. It appears as if $\delta + \sigma \sqrt{\log n}$ can be interpreted as a geometrical width of our problem. If we ask that $\delta$ is of the same order as $\sigma \sqrt{\log n}$, a maximum gaussian deviation for $n$ variables, then all conditions on $\delta$ from \eqref{sep4} can be removed. Thus for convenience of the following discussion we will now assume $\delta \lesssim \sigma \sqrt{\log n}$.

\noindent
How optimal is the result from Theorem \ref{thm:mainthm}? Notice that our result is adapted to anisotropy in the noise, but to discuss optimality it is easier to look at the isotropic scenario: $\mathcal{V}^2 = \sqrt{p} \sigma^2$ and $\gamma^2 = p \sigma^2$. Therefore $\Delta^2(\bmu)/\sigma^2$ represents a signal-to-noise ratio. For simplicity let us also assume that all groups have equal size, that is $|G_1|=...=|G_K|=m$ so that $n = mK$ and the sufficient condition \eqref{sep4} becomes
\begin{align}
\label{simpl1}
\frac{\Delta^2(\bmu)}{\sigma^2} \gtrsim \big(K + \log n\big) + \sqrt{ (K + \log n)\frac{pK}{n}}.
\end{align}

\noindent
\textbf{Optimality.} To discuss optimality, we distinguish between low and high dimensional setups.

\noindent
In the low-dimensional setup $n \vee m\log n \gtrsim p$, we obtain the following condition:
\begin{align}
	\label{lowdim}
	\frac{\Delta^2(\bmu)}{\sigma^2} \gtrsim \big(K + {\log n}\big).
\end{align}
Discriminating with high probability between $n$ observations from two gaussians in dimension 1 would require a separating rate of at least $\sigma^2 \log n$. This implies that when $K \lesssim \log n$, our result is minimax. Otherwise, to our knowledge the best clustering result on approximating mixture center is from \cite{mixon}, and on the condition that $\Delta^2(\bmu)/\sigma^2 \gtrsim K^2$. Furthermore, the $K \gtrsim \log n$ regime is known in the stochastic-block-model community as a hard regime where a gap is surmised to exist between the minimal information-theoretic rate and the minimal achievable computational rate (see for example \cite{chenxu}).

\noindent
In the high-dimensional setup $n \vee m\log n \lesssim p$, condition \eqref{simpl1} becomes:
\begin{align}
	\frac{\Delta^2(\bmu)}{\sigma^2} \gtrsim \sqrt{ (K + \log n)\frac{pK}{n}}.
\end{align}
There are few information-theoretic bounds for high-dimension clustering. Recently, Banks, Moore, Vershynin, Verzelen and Xu (2017) \cite{banks} proved a lower bound for Gaussian mixture clustering {detection}, namely they require a separation of order $\sqrt{K (\log K) p/n}$. When $K \lesssim \log n$, our condition is only different in that it replaces $\log (K)$ by $\log (n)$, a price to pay for going from detecting the clusters to exactly recovering the clusters. Otherwise when $K$ grows faster than $\log n$ there might exist a gap between the minimal possible rate and the achievable, as discussed previously.

\noindent
\textbf{Adaptation to effective dimension.} We can analyse further the condition \eqref{sep4} by introducing an effective dimension $r_*$, measuring the largest volume repartition for our variance-bounding matrices $\Sigma_1,...,\Sigma_n$. Let
\begin{align}
	r_* := \frac{\gamma^2}{\sigma^2} = \frac{\max_{a\in[n]} |\Sigma_{a}|_{*}}{\max_{a\in[n]} |\Sigma_{a}|_{op}},
\end{align}
$r_*$ can also be interpreted as a form of global effective rank of matrices $\Sigma_a$. Indeed, define $Re(\Sigma) := |\Sigma|_* / |\Sigma|_{op}$, then we have $r_* \leqslant \max_{a\in [n]} Re(\Sigma_a) \leqslant \max_{a\in [n]} \rank(\Sigma_a) \leqslant p$.

\noindent
Now using $\mathcal{V}^2 \leqslant \sqrt{r_*} \sigma^2$ and $\gamma = \sqrt{r_*} \sigma$, condition \eqref{sep4} can be written as
\begin{align}
\label{simplified}
\frac{\Delta^2(\bmu)}{\sigma^2} \gtrsim \big(K + \log n\big) + \sqrt{ (K + \log n)\frac{r_*K}{n}}.
\end{align}

\noindent
By comparing this equation to \eqref{simpl1}, notice that $r_*$ is in place of $p$, indeed playing the role of an effective dimension for the problem. This also shows that our estimator adapts to this effective dimension, without any dimension reduction step. In consequence, equation \eqref{simplified} distinguishes between an actual high-dimensional setup: $n \vee m\log n \lesssim r_*$ and a "low" dimensional setup $r_* \lesssim n \vee m\log n$ under which, regardless of the actual value of $p$, our estimators recovers under the near-minimax condition of \eqref{lowdim}.

\noindent
This informs on the effect of correcting term $\widehat{\Gamma}^{corr}$ in the theorem above when $n + m\log n \lesssim r_*$. The un-corrected version of the semi-definite program \eqref{sdp} has a leading separating rate of $\gamma^2/m = \sigma^2 r_*/m$, but with the $\widehat{\Gamma}^{corr}$ correction on the other hand, \eqref{simplified} has leading separating factor smaller than $\sigma^2 \sqrt{ (K + \log n){r_*}/{m}} = \sigma^2 \sqrt{n + m\log n} \times  \sqrt{r_*}/m$. This proves that in a high-dimensional setup, our correction enhances the separating rate of at least a factor $\sqrt{(n + m\log n) / r_*}$.

\section{Adaptation to the unknown number of group $K$}
It is rarely the case that $K$ is known, but we can proceed without it. We produce an estimator adaptive to the number of groups $K$:  let $\widehat{\kappa} \in \mathbb{R}_+$, we now study the following adaptive estimator:
\begin{align}
\widetilde{B}^{corr} := \argmax_{B \in \mathcal{C}} \langle \widehat{\Lambda} - \widehat{\Gamma}^{corr},B \rangle - \widehat{\kappa}\tr(B).
\end{align}

\begin{theorem}
	\label{thm:mainthmadapt}
	Suppose that $m>2$ and $\eqref{sep4}$ is satisfied. For $c_3,c_4,c_5>0$ absolute constants suppose that the following condition on $\widehat{\kappa}$ is satisfied
	\begin{align}
	\label{kappa}
	c_4 \Big(\mathcal{V}^2 \sqrt{n} + \sigma^2 n + \gamma(\sigma\sqrt{\log n} + \delta) + \delta^2 \sqrt{n} \Big) < c_5 \widehat{\kappa} <  m\Delta^2(\bmu),
	\end{align}
	then we have $\widetilde{B}^{corr} = B^*$ with probability larger than $1-c_3/n$
\end{theorem}

\noindent
Notice that condition \eqref{kappa} essentially requires $\widehat{\kappa}$ to be seated between $m \Delta^2(\bmu)$ and some components of the right-hand side of \eqref{sep4}. So under \eqref{kappa}, the results from the previous section apply to the adaptive estimator $\widetilde{B}^{corr}$ as well and this shows that it is not necessary to know $K$ in order to perform well for recovering $\mathcal{G}$. Finding an optimized, data-driven parameter $\widehat{\kappa}$ using some form of cross-validation is outside of the scope of this paper.

\section{Numerical experiments}

We illustrate our method on simulated Gaussian data in two challenging, high-dimensional setup experiments for comparing clustering estimators. Our sample are drawn from $K=3$ identically-sized, identically distributed and perfectly discriminated clusters of non-isovolumic Gaussians. The distributions are chosen to be isotropic, and the ratio between the lowest and the highest standard deviation is of 1 to 10. We draw points of a $\mathbb{R}^p$ space in two different scenarii. In $(\mathcal{S}_1)$, for a given dimension space $p=2000$ and a fixed isotropic noise level, we report the algorithms' compared performances as the signal-to-noise ratio $\Delta^2(\bmu)/\sigma^2$ is increased from 1 to 20. In $(\mathcal{S}_2)$ we impose a fixed signal to noise ratio, and observe the algorithm's decay in performance as the space dimension $p$ is increased from 100 to 400 000. All points of the simulated space are reported as a median value with asymmetric standard deviations in the form of errorbars over a hundred simulations.

\noindent
Solving for estimator $\widehat{B}^{corr}$ is a hard problem as $n$ grows. For this task we implemented an ADMM solver from the work of Boyd \textit{et al.} \cite{admm} with multiple stopping criterions including a fixed number of iterations of $T=3000$. The results we report use $n = 30$ samples. For reference, we compare the recovering capacities of $\widehat{\mathcal{G}}^{corr}$, labeled 'pecok' in Figure \ref{fig:fig} with other classical clustering algorithm. We chose three different but standard clustering procedures: Lloyd's K-means algorithm \cite{lloyd} with K-means++ initialization \cite{arthur} (although in scenario $(\mathcal{S}_2)$, it is too slow to converge as $p$ grows so we do not report it), Ward's method for Hierarchical Clustering \cite{ward} and the low-rank clustering algorithm applied to the Gram matrix, a spectral method appearing in McSherry \cite{lowrank}. Lastly we include the CORD algorithm from Bunea \textit{et al.} \cite{cord}.
 
\begin{figure}
	\begin{subfigure}{.5\textwidth}
		\centering
		\includegraphics[width=\linewidth]{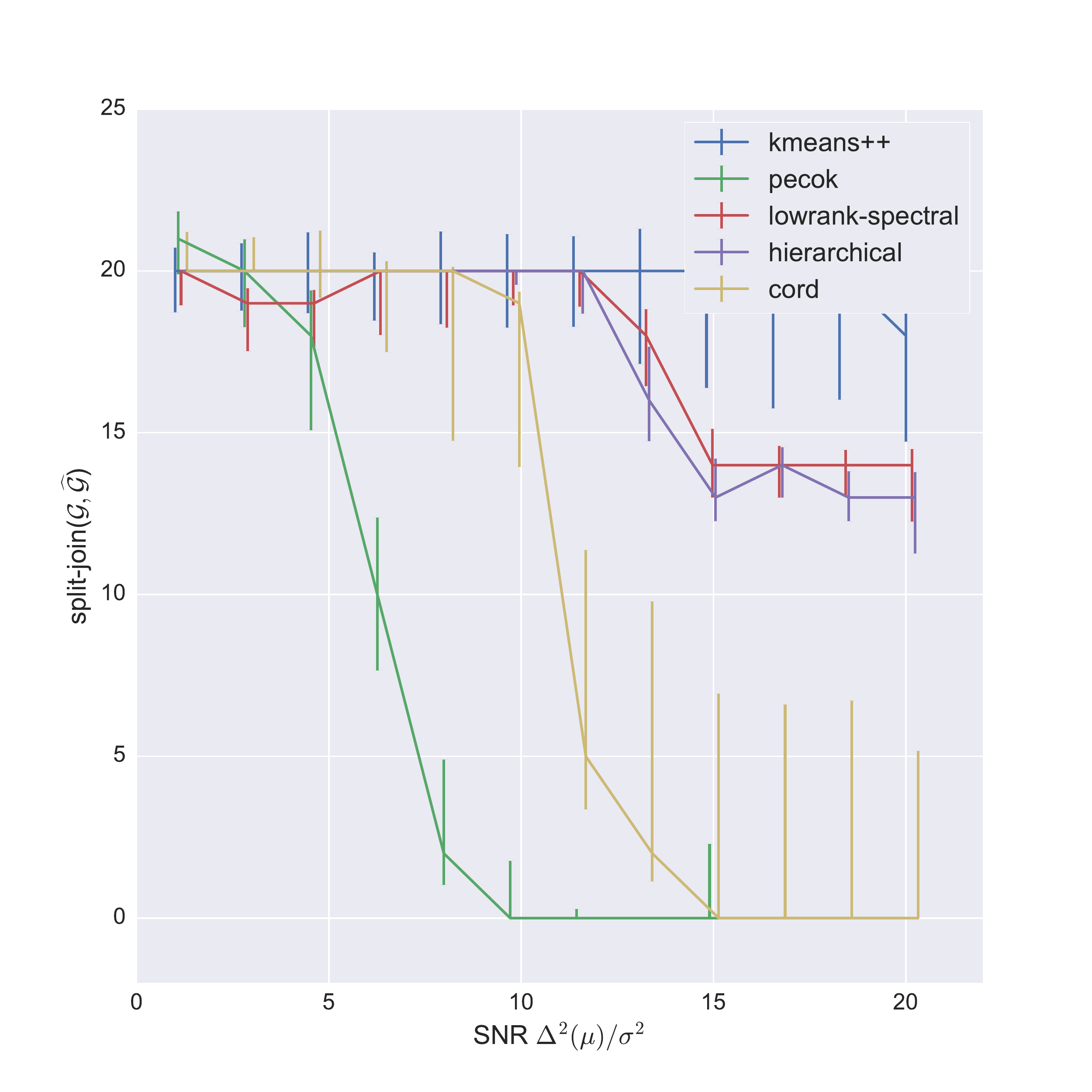}
		\caption*{Scenario $(\mathcal{S}_1)$ }
		\label{fig:compare2}
	\end{subfigure}%
	\begin{subfigure}{.5\textwidth}
		\centering
		\includegraphics[width=\linewidth]{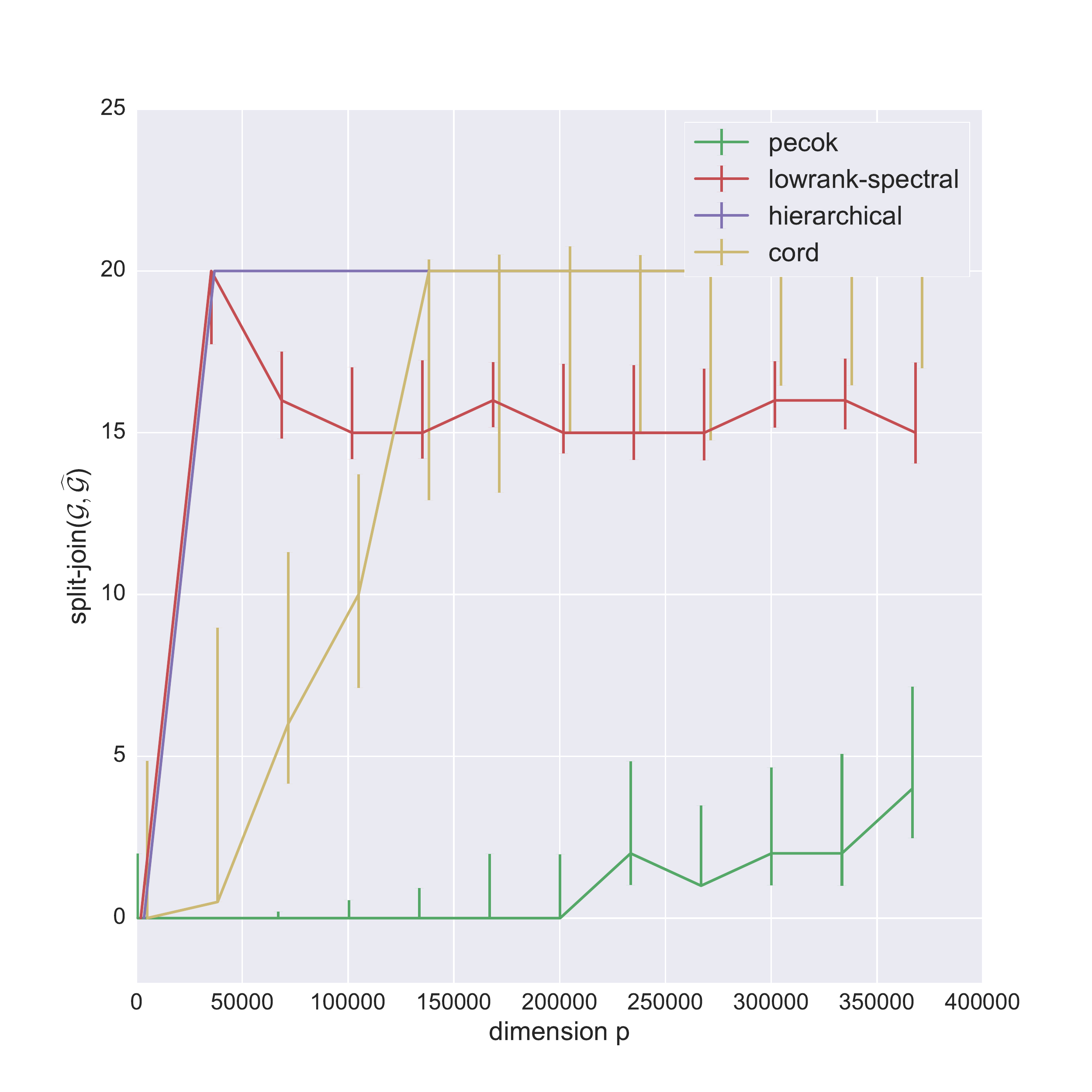}
		\caption*{Scenario $(\mathcal{S}_2)$}
		\label{fig:compare1}
	\end{subfigure}
	\caption{Performance comparison for classical clustering estimators and ours $\widehat{\mathcal{G}}^{corr}$, labeled 'pecok' in reference to \cite{pecok}. The lower split-join, the better the clustering performance and split-join$(\mathcal{G}, \widehat{\mathcal{G}}) = 0$ implies $\widehat{\mathcal{G}} = \mathcal{G}$.}
	\label{fig:fig}
\end{figure}

\noindent
We measure the performances of estimators by computing the split-join metric on the cluster graphs, counting the number of edges to remove or add to go from one graph to the other. In the two experiments, the results of $\widehat{\mathcal{G}}^{corr}$ are markedly better than that of other methods. Scenario $(\mathcal{S}_1)$ shows it can achieve exact recovery with a lesser signal to noise ratio than its competitors, whereas scenario $(\mathcal{S}_2)$ shows its performances are decaying at a much lower rate than the others when the space dimension is increased.

\noindent
Because of the slow convergence of ADMM, $\widehat{\mathcal{G}}^{corr}$ comes with important computation times. Of course all of the compared methods have a very hard time reaching high sample sizes $n$ in the high dimensional context and to that regard, the low-rank clustering method is by far the most promising.

\section{Conclusion}

In this paper we analyzed a new semidefinite positive algorithm for clustering within the context of a flexible probabilistic model and exhibit the key quantities that guarantee non-asymptotic exact recovery. It implies an essential bias-removing correction that significanty improves the recovering rate in the high-dimensional setup. Hence we showed the estimator to be near-minimax, adapted to an effective dimension of the problem. We demonstrated that our estimator can in theory be optimally adapted to a data-driven choice of $K$. Lastly we illustrated on high-dimensional experiments that our approach is empirically stronger than other classical clustering methods.

\noindent
Our method is computationally intensive even though it is of polynomial order. As the $\widehat{\Gamma}^{corr}$ correction step of the algorithm can be interpreted as an independent, denoising step for the Gram matrix, we suggest using it as such for other notably faster algorithm such as the spectral algorithms.

\subsection*{Acknowledgements}
This work is supported by a public grant overseen by the French National research Agency (ANR) as part of the ``Investissement d'Avenir" program, through the ``IDI 2015" project funded by the IDEX Paris-Saclay, ANR-11-IDEX-0003-02. It is also supported by the CNRS PICS funding HighClust. We thank Christophe Giraud for a shrewd, unwavering thesis direction.

\bibliographystyle{plain} 
\bibliography{biblio} 

\pagebreak

\section*{Appendix}

\begin{alphasection}
	
\section{Intermediate results}

\subsection{Generic controls for exact recovery}

Let $\widehat{\Gamma}$ be any estimator of $\Gamma$ and let $\widehat{B} := \argmax_{B \in \mathcal{C}_K} \langle \widehat{\Lambda} - \widehat{\Gamma},B \rangle$.

\begin{theorem}
	\label{thm:exactrecovery}
	For $c_1,c_2>0$ absolute constants suppose that $|\widehat{\Gamma} - \Gamma|_{V} \leqslant \bar{\gamma}^2_{n}$ with probability $1-c_1/n$, and that
	\begin{align}
	\label{sep}
	m\Delta^2(\bmu) &\geqslant c_2 \Big( \sigma^2 ( n + m\log n ) + \mathcal{V}^2 ( \sqrt{n + m\log n} ) + \bar{\gamma}^2_{n} + \delta^2 (\sqrt{n}+m) \Big),
	\end{align}
	then we have $\widehat{B} = B^*$ with probability larger than $1-c_1/n$
\end{theorem}

\noindent
In the case where the number of groups is unknown we study $\widetilde{B} := \argmax_{B \in \mathcal{C}} \langle \widehat{\Lambda} - \widehat{\Gamma},B \rangle - \widehat{\kappa}\tr(B)$ for $\widehat{\kappa} \in \mathbb{R}$.

\begin{theorem}
	\label{thm:exactrecoveryadapt}
	For $c_3,c_4,c_5>0$ absolute constants suppose that $|\widehat{\Gamma} - \Gamma|_{\infty} \leqslant \bar{\gamma}^2_{n}$ with probability $1-c_3/n$. Suppose that $\eqref{sep}$ is satisfied and that the following condition on $\widehat{\kappa}$ is satisfied
	\begin{align}
	\label{condition22}
	c_4 \Big( \mathcal{V}^2 \sqrt{n} + \sigma^2 n + \bar{\gamma}_n^2 + \delta^2 \sqrt{n}\Big) < c_5 \widehat{\kappa} <  m\Delta^2(\bmu),
	\end{align}
	then we have $\widetilde{B} = B^*$ with probability larger than $1-c_3/n$
\end{theorem}

\subsection{On estimating $\Gamma$}

In the general case we have $\widehat{\Gamma} = 0$ hence a deterministic perturbation term $\bar{\gamma}_n^2 = |\Gamma|_\infty$ weighing on the separation requirements. For $\widehat{\Gamma}^{corr}$, we have the following result.
\begin{proposition}
	\label{prop:gamma}
	Assume that $m > 2$. For $c_6,c_7>0$ absolute constants, with probability larger than $1-c_6/n$ we have
	\begin{align}
	|\widehat{\Gamma}^{corr}-\Gamma|_\infty \leqslant c_7 \Big( \sigma^2 {\log n} + (\delta + \sigma\sqrt{\log n})\gamma + \delta^2 \Big).
	\end{align}
\end{proposition}

\subsection{Concentration of random subgaussian Gram matrices}

A key result in our proof is the following concentration bound on the Gram matrix of centered, subgaussian, independent random variables.
\begin{lemma}
	\label{lem:martsubG}
	For some absolute constant $c_*>0$, for $a\in[n]$ let $E_a$ be centered, independent random vectors in $\mathbb{R}^d$, $E_a\sim$ subg$(\Sigma_a)$. Let $\bE := \left[ \begin{smallmatrix} ... \\ E_a^T \\ ... \end{smallmatrix} \right] \in \mathbb{R}^{n\times d}$ then $\forall t \geqslant 0$
	\begin{align}
	\proba{  |\bE \bE^T - \expect{\bE \bE^T}|_{op} \geqslant 2\max_{a\in[n]} |\Sigma_a|_F \sqrt{t}+2\max_{a\in[n]} |\Sigma_a|_{op} t } \leqslant 9^{n} 2e^{-c_*t}.
	\end{align}
\end{lemma}

\section{Main proofs}
\subsection{Proof of Proposition \ref{prop:identifiability}: identifiability}
Suppose that $X_1,...,X_n$ are $(\mathcal{G},\bmu,\delta)$-clustered with $|\mathcal{G}| = K$, and $\rho(\mathcal{G},\bmu,\delta) > 4$. Then we remark that for $(a,b) \in [n]^2$, $a \overset{\mathcal{G}}{\sim} b$ is equivalent to $|\nu_a - \nu_b|_2 \leqslant 2\delta$ because:
\begin{itemize}
	\item if $a \overset{\mathcal{G}}{\sim} b$ then there exist $k \in [K]$ such that $|\nu_a - \nu_b|_2 \leqslant |\nu_a - \mu_k|_2 + |\mu_k - \nu_b|_2 \leqslant 2 \delta$
	\item if $a \overset{\mathcal{G}}{\not\sim} b$ then there exist $(k, l) \in [K]^2$ such that $|\nu_a - \nu_b|_2 \geqslant |\mu_k - \mu_l|_2 - |\nu_a - \mu_k|_2 - |\nu_b - \mu_l|_2 > 4 \delta - 2 \delta > 2 \delta$.
\end{itemize}

\noindent
Now suppose there exist $\mathcal{G}'$ such that $X_1,...,X_n$ are $(\mathcal{G}',\bmu',\delta')$-clustered with $|\mathcal{G}'| = K$ and $\rho(\mathcal{G}',\bmu',\delta') > 4$. By symmetry we can assume $\delta' \leqslant \delta$, and the previous remark shows that $\mathcal{G}'$ is a sub-partition of $\mathcal{G}$, ie $\mathcal{G}$ preserves the structure of $\mathcal{G}'$. But since $|\mathcal{G}| = |\mathcal{G}'|$ this implies $\mathcal{G} = \mathcal{G}'$. \qed

\subsection{Exact recovery with high probability}

The proof for Theorem \ref{thm:mainthm} (respectively Theorem \ref{thm:mainthmadapt}) is a composition of Theorem \ref{thm:exactrecovery} (respectively Theorem \ref{thm:exactrecoveryadapt}) and Proposition \ref{prop:gamma}.

\noindent
In this section, under Hypothesis \ref{hyp:model}, we have $\forall k \in [K], \forall a \in G_k: X_a \sim$ subg$(\Sigma_a)$. For $k \in [K]$, we define
$\sigma_k^2 := \max_{a\in G_k} |\Sigma_{a}|_{op} \leqslant \sigma^2, \mathcal{V}_k^2 := \max_{a\in G_k} |\Sigma_{a}|_{F} \leqslant \mathcal{V}^2, \gamma_k^2 := \max_{a\in G_k} |\Sigma_{a}|_{*} \leqslant \gamma^2$.

\noindent
A number of proofs in this section are adapted from the proof ensemble of \cite{pecok}. In it the authors use a latent model for variable clustering. A comparable model in this work would require to impose the following conditions on $X_1,...,X_n$: identically distributed variables within a group (implying $\delta = 0$) and isovolumic, Gaussian distributions.

\subsubsection{Proof of Theorem \ref{thm:exactrecovery}}
In this theorem we only need to consider $B\in \mathcal{C}_K$, but the proof of Theorem \ref{thm:exactrecoveryadapt} is similar to this one, hence we will start by considering the more general $B\in \mathcal{C}$ and use $B\in \mathcal{C}_K$ at a later stage of the proof. Thus we want to prove that under some conditions, with high probability:
\begin{align}
\label{eq:scalaroptim}
\langle \widehat{\Lambda} - \widehat{\Gamma},B^* - B \rangle > 0 \text{ for all } B\in \mathcal{C} \setminus \{B^*\}
\end{align}

\noindent
For $(a,b)\in G_k \times G_l$ for $(k,l) \in [K]^2$, let:
\begin{align}
	(S_1)_{ab} &:= -|\mu_{k}-\mu_{l}|_2^2/2\\
	(W_1)_{ab} &:= \langle \nu_a-\mu_k,\nu_b-\mu_l \rangle \nonumber\\
	(W_2)_{ab} &:= \langle \mu_k-\nu_a + \nu_b-\mu_l + E_b - E_a, \mu_{k}-\mu_{l} \rangle \nonumber\\
	(W_3)_{ab} &:= \langle E_{b}-E_{a} , \nu_a - \mu_k + \mu_l - \nu_b \rangle \nonumber\\
	(W_4)_{ab} &:= (\langle E_a ,E_b\rangle-{\Gamma}_{ab}) \nonumber\\
	(W_5)_{ab} &:= (\Gamma-\widehat{\Gamma})_{ab} \nonumber
\end{align}

\begin{lemma}
	\label{lem:decomp}
	Proving \eqref{eq:scalaroptim} reduces to proving
	\begin{align}
	\label{eq:decomposition}
	\langle S_1 + W_1 + W_2 + W_3 + W_4 + W_5,B^* - B \rangle > 0 \text{ for all } B\in \mathcal{C} \setminus \{B^*\}.
	\end{align}
\end{lemma}

\noindent
The proof for \lemref{decomp} is found in section \ref{sec:lemdecomp}. So we need only concern ourselves with the quantities $S_1,W_1,W_2,W_3,W_4, W_5$. The term $S_1$ contains our uncorrupted signal and since $\langle S_1, B^* \rangle = 0$ it writes:
\begin{align}
\label{S1}
\langle S_1,B^* - B \rangle = \sum_{1 \leqslant k\neq l \leqslant K} \frac{1}{2} | \mu_k - \mu_l |_2^2 |B_{G_k G_l} |_1
\end{align}
The other parts are noisy and must be controlled. The term $W_2$ is a simple subgaussian form controlled through the following lemma, proved in section \ref{sec:lemma2}:
\begin{lemma}
	\label{lem:w2}
	For $c'_2>0$ absolute constant, with probability greater than $1-1/n$:
	\begin{align}
	\forall B \in \mathcal{C}, \quad |\langle W_2, B^*-B \rangle | \leqslant \sum_{1 \leqslant k\neq l \leqslant K} \Big(2\delta +\sqrt{c'_2 (\log n) (\sigma_k^2 + \sigma_l^2)} \Big) |\mu_k - \mu_l|_2 |B_{G_k  G_l} |_1.
	\end{align}
\end{lemma}

\noindent
To control the other noisy terms we now introduce a deterministic result:
\begin{lemma}
	\label{lem:deter}
	For any symmetric matrix $W \in \mathbb{R}^{n\times n}$ we have:
	\begin{align}
	\label{opinf}
	\forall B \in \mathcal{C}, \quad |\langle W, B^* - B\rangle| \leqslant 6|B^*W|_\infty\sum_{1 \leqslant k\neq l \leqslant K} |B_{G_k G_l}|_1 + | W |_{op}  \Big[ {\sum_{1 \leqslant k\neq l \leqslant K} |B_{G_k G_l}|_1}/m + (\tr(B)-K) \Big].
	\end{align}
\end{lemma}

\noindent
The proof for \lemref{deter} will  be found in \cite{pecok}, p.21-22 until eq. (58).

\noindent
As $B^* 1 = 1$ and $B^* \geqslant 0$, $|B^*W|_\infty \leqslant |W|_\infty$ so we use the lemma on terms $W_1$ and $W_3$ by bounding $|W|_{\infty}$ and $|W|_{op}$: for the term $W_1$ we use $|W_1|_{\infty} \leqslant \delta^2$ so $|W_1|_{op} \leqslant \delta^2 \sqrt{n}$. To control the term $W_3$, we use the subgaussian tail bound of \eqref{tail} with $|\nu_a - \mu_k + \mu_l - \nu_b|_2 \leqslant 2\delta$ and a union bound over $(a,b) \in [n]^2$. We get that for $c'_3>0$ absolute constant, with probability greater than $1-1/n$, $|W_3|_\infty \leqslant \sqrt{c'_3 (\log n) \sigma^2 \delta^2}$ and $|W_3|_{op} \leqslant \sqrt{c'_3 (\log n) \sigma^2 \delta^2}\times \sqrt{n}$ therefore with probability greater than $1-1/n$, $\forall B \in \mathcal{C}$:
\begin{align}
\label{W1}
|\langle W_1, B^* - B\rangle| \leqslant \delta^2 \Big[ \sum_{1 \leqslant k\neq l \leqslant K} |B_{G_k G_l}|_1 (6 + \frac{\sqrt{n}}{m}) + \sqrt{n} (\tr(B)-K)_+ \Big]\\
|\langle W_3, B^* - B\rangle| \leqslant \sqrt{c'_3 (\log n) \sigma^2  \delta^2} \Big[ \sum_{1 \leqslant k\neq l \leqslant K} |B_{G_k G_l}|_1 (6 + \frac{\sqrt{n}}{m}) + \sqrt{n} (\tr(B)-K)_+ \Big]
\end{align}
For the term $W_4$ we introduce the following lemma, proved in section \ref{sec:lemma4}:
\begin{lemma}
	\label{lem:w3}
	For $c'_4,c''_4>0$ absolute constants, with probability larger than $1-2/n$:
	\begin{align}
	\forall B \in \mathcal{C}, \quad |\langle W_4, B^* - B\rangle| \leqslant &\Big[ {6c'_4}(\mathcal{V}^2 \sqrt{\log n} + \sigma^2 \log n)/\sqrt{m} + c''_4(\mathcal{V}^2 \sqrt{n} + \sigma^2 n)/m \Big] \sum_{1 \leqslant k\neq l \leqslant K} |B_{G_k G_l}|_1 \nonumber \\&\qquad+ (\tr(B)-K)_+ c''_4(\mathcal{V}^2 \sqrt{n} + \sigma^2 n).
	\end{align}
\end{lemma}

\noindent
Lastly as the term $W_5$ is diagonal we have $|W_5|_{op} = |W_5|_{\infty}$ and $|B^*W_5|_{\infty} \leqslant |W_5|_{\infty}/m$ therefore:
\begin{align}
\label{genW5}
\forall B \in \mathcal{C}, \quad |\langle W_5, B^* - B\rangle| \leqslant |W_5|_{\infty} \Big[ \frac{7}{m} \sum_{1 \leqslant k\neq l \leqslant K} |B_{G_k G_l}|_1  + (\tr(B)-K)_+ \Big]
\end{align}

\noindent
Using those controls of $W_1,W_2,W_3,W_4,W_5$, in combination in a union bound in \eqref{eq:decomposition} we get for $c'_1>0$ absolute constant, with probability greater than $1-c'_1/n$: $\forall B \in \mathcal{C}$,
\begin{align} \label{sumW}
&\langle S_1 + W_1 + W_2 + W_3 + W_4 + W_5, B^* - B \rangle \geqslant \sum_{1 \leqslant k\neq l \leqslant K} \Big[ \frac{1}{2} | \mu_k - \mu_l |^2_2 - \Big(2\delta + \sqrt{2c'_2 (\log n) \sigma^2}\Big) |\mu_k - \mu_l |_2 \nonumber\\
& - ({6c'_4}\frac{\mathcal{V}^2 \sqrt{\log n} + \sigma^2 \log n}{\sqrt{m}} + c''_4\frac{\mathcal{V}^2 \sqrt{n} + \sigma^2 n}{m}) - \frac{7}{m}|W_5|_\infty - (6+\frac{\sqrt{n}}{m})(\delta^2 + \sqrt{c'_3 (\log n) \sigma^2 \delta^2}) \Big]|B_{G_k G_l}|_1 \nonumber \\
&- (\tr(B)-K)_+ [c''_4 (\mathcal{V}^2 \sqrt{n} + \sigma^2 n) + (\delta^2 + \sqrt{c'_3 (\log n) \sigma^2 \delta^2}){\sqrt{n}} + |W_5|_\infty]
\end{align}

\color{black}
\noindent
We now use the fact that for this theorem we are only considering $B \in \mathcal{C}_K$, ie matrices such that $\tr(B)=K$ so we can discard the last line of \eqref{sumW}. In this particular context we can improve the control provided by \lemref{deter} for $W_5$: as $\tr(B^*) = K$, we have for $\alpha \in \mathbb{R}: |\langle W_5, B^* - B\rangle| \leqslant |\langle W_5 - \alpha I_n, B^* - B\rangle | + |\alpha(\tr(B)-K)|$. So by choosing $\alpha = (\max_a (W_5)_{aa}+\min_a (W_5)_{aa})/2$, we have $|W_5-\alpha I_n|_{op} = |W_5 -\alpha I_n|_{\infty} = |W_5|_V/2$ and therefore:
\begin{align}
\label{W5K}
\forall B \in \mathcal{C}_K \quad |\langle W_5, B^* - B\rangle| \leqslant |W_5|_V \frac{7}{2m} \sum_{1 \leqslant k\neq l \leqslant K} |B_{G_k G_l}|_1
\end{align}
In consequence we can replace $|W_5|_\infty$ by $|W_5|_V/2$ in the second line of \eqref{sumW}, and with another union bound, by assumption we replace $|W_5|_V/2$ by $\bar{\gamma}_n^2/2$.

\noindent
Lastly Lemma 3 p. 17 from \cite{pecok} shows the only matrix in $\mathcal{C}_K$ whose support is included in $supp(B^*)$ is $B^*$, therefore $B \in \mathcal{C}_K \setminus \{B^*\}$ implies $\sum_{1 \leqslant k\neq l \leqslant K}|B_{G_k G_l}|_1 > 0$. Hence for $c_2>0$ absolute constant, the following condition on $\Delta(\bmu)$ is sufficient to ensure exact recovery with probability larger than $1-c_1/n$:
\begin{align}
\Delta^2(\bmu) \geqslant c_2 \big[ \sigma^2 m\log n + \mathcal{V}^2 \sqrt{{m\log n}} + \mathcal{V}^2 {\sqrt{n}} + \sigma^2 {n} + \bar{\gamma}_n^2 +  \delta^2 (\sqrt{n}+m) \big] \times \frac{1}{m}
\end{align}
This concludes the proof for Theorem \ref{thm:exactrecovery}.\qed

\subsubsection{Proof of Theorem \ref{thm:exactrecoveryadapt}: adaptive exact recovery}
In this Theorem we need to take into account the additional penalization term $\widehat{\kappa}\tr(B)$. Notice it is equivalent to a correction by $\widehat{\kappa}I_n$ of our estimator $\widehat{\Lambda} - \widehat{\Gamma}$, therefore for $B \in \mathcal{C}$, $\langle \widehat{\Lambda} - \widehat{\Gamma} - \widehat{\kappa}I_n,B^* - B \rangle = \langle \widehat{\Lambda} - \widehat{\Gamma},B^* - B \rangle + \widehat{\kappa} \times (\tr(B) - K)$. Therefore for Theorem \ref{thm:exactrecoveryadapt} we can follow the same proof as in Theorem \ref{thm:exactrecovery} until establishing \eqref{sumW}, at which point we can use a union bound to use the assumption $|W_5|_\infty \leqslant \bar{\gamma}_n^2$. Consequently we have with probability greater than $1-c'_1/n$: $\forall B \in \mathcal{C}$,
\begin{align}
&\langle S_1 + W_1 + W_2 + W_3 + W_4 + W_5, B^* - B \rangle \geqslant \sum_{1 \leqslant k \neq l \leqslant K} \Big[ \frac{1}{2} | \mu_k - \mu_l |^2_2 - \Big(2\delta + \sqrt{2c'_2 (\log n) \sigma^2}\Big) |\mu_k - \mu_l |_2 \nonumber\\
& - ({6c'_4}\frac{\mathcal{V}^2 \sqrt{\log n} + \sigma^2 \log n}{\sqrt{m}} + c''_4\frac{\mathcal{V}^2 \sqrt{n} + \sigma^2 n}{m}) - \frac{7}{m} \bar{\gamma}_n^2 - (6+\frac{\sqrt{n}}{m})(\delta^2 + \sqrt{c'_3 (\log n) \sigma^2 \delta^2}) \Big]|B_{G_k G_l}|_1 \nonumber \\
&- (\tr(B)-K)_+ [c''_4 (\mathcal{V}^2 \sqrt{n} + \sigma^2 n) + (\delta^2 + \sqrt{c'_3 (\log n) \sigma^2 \delta^2}){\sqrt{n}} + \bar{\gamma}_n^2] + \widehat{\kappa}(\tr(B)-K)
\end{align}

\noindent
Using the assumption \eqref{sep} of Theorem \ref{thm:exactrecoveryadapt} there exist $c'_2>0$ such that with probability greater than $1-c'_1/n$: $\forall B \in \mathcal{C}$,
\begin{align}
\langle &S_1 + W_1 + W_2 + W_3 + W_4,B^* - B \rangle \geqslant c'_2 \Delta^2(\bmu) \sum_{1\leqslant k \neq l \leqslant K} |B_{G_k G_l}|_1 \nonumber\\
&- (\tr(B)-K)_+ [c''_4 (\mathcal{V}^2 \sqrt{n} + \sigma^2 n) + (\delta^2 + \sqrt{c'_3 (\log n) \sigma^2 \delta^2}){\sqrt{n}} + \bar{\gamma}_n^2] + \widehat{\kappa}(\tr(B)-K)
\end{align}

\noindent
From here, when $\tr(B) > K$, the left-hand side of \eqref{condition22} is sufficient to ensure recovery. When $\tr(B) = K$, we already established that $\sum_{1\leqslant k \neq l \leqslant K} |B_{G_k G_l}|_1>0$ for all matrices $B \in \mathcal{C}_K \setminus \{B^*\}$ so \eqref{sep} is sufficient in that case. Lastly note that $K-\tr(B) \leqslant \frac{1}{m} \sum_{1\leqslant k \neq l \leqslant K}|B_{G_k G_l}|_1$ (see \cite{pecok} eq. (57) p.21) so the right-hand side of \eqref{condition22} is sufficient condition for recovery when $\tr(B)-K < 0$. This concludes the proof of Theorem \ref{thm:exactrecoveryadapt}. \qed

\subsubsection{Proof of Lemma \ref{lem:decomp}}
\label{sec:lemdecomp}
\begin{align}
(\widehat{\Lambda} - \widehat{\Gamma})_{ab} &= \langle X_a, X_b \rangle - \widehat{\Gamma}_{ab} = \langle \nu_a, \nu_b \rangle + \langle \nu_a, E_b \rangle + \langle \nu_b, E_a \rangle + \langle E_a, E_b \rangle - \widehat{\Gamma}_{ab}\\
&= \langle \nu_a, \nu_b \rangle + \langle \nu_a-\nu_b, E_b-E_a \rangle + \langle \nu_a, E_a \rangle + \langle \nu_b, E_b \rangle + (W_4+W_5)_{ab}\\
&= \langle \nu_a, \nu_b \rangle + \langle \mu_k-\mu_l, E_b-E_a \rangle + (W_3)_{ab} + \langle \nu_a, E_a \rangle + \langle \nu_b, E_b \rangle + (W_4+W_5)_{ab}\\
&= -\langle \mu_k, \mu_l \rangle + \langle \nu_a-\mu_k, \nu_b-\mu_l \rangle + \langle \nu_a, \mu_l \rangle + \langle \mu_k, \nu_b \rangle \nonumber \\ &\quad+ \langle \mu_k-\mu_l, E_b-E_a \rangle + (W_3)_{ab} + \langle \nu_a, E_a \rangle + \langle \nu_b, E_b \rangle + (W_4+W_5)_{ab}\\
&= -(S_1)_{ab} - \frac{1}{2}(|\mu_k|_2^2 + |\mu_l|_2^2) + (W_1)_{ab} + \langle \nu_a, \mu_l \rangle + \langle \mu_k, \nu_b \rangle \nonumber \\ &\quad+ \langle \mu_k-\mu_l, E_b-E_a \rangle + (W_3)_{ab} + \langle \nu_a, E_a \rangle + \langle \nu_b, E_b \rangle + (W_4+W_5)_{ab}\\
&= -(S_1)_{ab} - \frac{1}{2}(|\mu_k|_2^2 + |\mu_l|_2^2) + (W_1)_{ab} + \langle \nu_a, \mu_k \rangle + \langle \mu_l, \nu_b \rangle \nonumber \\ &\quad+ \langle \mu_k-\mu_l, \nu_b - \nu_a + E_b-E_a \rangle + (W_3)_{ab} + \langle \nu_a, E_a \rangle + \langle \nu_b, E_b \rangle + (W_4+W_5)_{ab}\\
&= -(S_1)_{ab} - \frac{1}{2}(|\mu_k|_2^2 + |\mu_l|_2^2) + (W_1)_{ab} + \langle \nu_a, \mu_k \rangle + \langle \mu_l, \nu_b \rangle \nonumber \\ &\quad+ 2(S_1)_{ab} + (W_2)_{ab} + (W_3)_{ab} + \langle \nu_a, E_a \rangle + \langle \nu_b, E_b \rangle + (W_4+W_5)_{ab}
\end{align}

\noindent
Now since $(\langle \nu_a, \mu_k \rangle)_{(a,b)\in [n]^2} = (\langle \nu_a, \mu_k \rangle)_{a\in [n]}\times 1^T_n$, $(|\mu_k|_2^2)_{(a,b)\in [n]^2} = (|\mu_k|_2^2)_{a\in [n]}\times 1^T_n$, $(\langle \nu_b, \mu_l \rangle)_{(a,b)\in [n]^2} = 1_n \times (\langle \nu_b, \mu_l \rangle)_{b\in [n]}$, $(|\mu_l|_2^2)_{(a,b)\in [n]^2} = 1_n \times (|\mu_l|_2^2)_{b\in [n]}$, $(\langle \nu_a,E_a \rangle)_{(a,b)\in [n]^2} = (\langle \nu_a,E_a \rangle)_{a\in [n]}\times 1^T_n$, $(\langle \nu_b,E_b \rangle)_{(a,b)\in [n]^2}$ $= 1_n \times (\langle \nu_b,E_b \rangle)_{b\in [n]}$ and since $B 1_n = B^* 1_n = (1_n^T B)^T = (1_n^T B^{*})^T = 1_n$, we have:
\begin{align}
\langle \widehat{\Lambda} - \widehat{\Gamma},B^* - B \rangle = \langle S_1 + W_1 + W_2 + W_3 + W_4 + W_5,B^* - B \rangle
\end{align} \qed

\subsubsection{Proof of Lemma \ref{lem:w2}: control of $|\langle W_2, B^*-B\rangle|$}
\label{sec:lemma2}
By definition, $(W_2)_{ab}=0$ when $k=l$ and $(B^*)_{ab}=0$ when $k \neq l$ so we have $\langle W_2 , B^* \rangle = 0$. Let $\langle A, B \rangle_{G_k G_l} = \sum_{(a,b) \in G_k \times G_l} A_{ab}B_{ab}$, we have:
\begin{align}
\langle W_2, B^*-B \rangle &= - \langle W_2, B \rangle = - \sum_{1 \leqslant k\neq l \leqslant K} \langle W_{2}, B \rangle_{G_k G_l} \leqslant \sum_{1 \leqslant k\neq l \leqslant K} |W_{2|G_k G_l}|_{\infty} |B_{G_k G_l}|_1
\end{align}
Let $(a, b)\in G_k\times G_l$, we look at $(W_2)_{ab} = \langle E_b-E_a -(\nu_a-\mu_k) + (\nu_b-\mu_l), \mu_{k}-\mu_{l} \rangle = \langle E_{a}-E_{b}, \mu_{k}-\mu_{l} \rangle + \langle -(\nu_a-\mu_k) + (\nu_b-\mu_l), \mu_{k}-\mu_{l} \rangle$. The term on the right is a constant offset bounded by $2\delta |\mu_k-\mu_l|_2$. \color{black}Let $z := \mu_{k}-\mu_{l}$, by \lemref{subg} $\langle E_{a}-E_{b}, z \rangle$ is a subgaussian variable with variance bounded by $(\sigma_k^2 + \sigma_l^2) |z|_2^2$ therefore its tails are characteristically bounded (see for example \cite{vershynin}), there exist $c_*>0$ absolute constant such that $\forall t \geqslant 0$:
\begin{align}
	\label{tail}
	\proba{ |\langle E_b-E_a, z \rangle| \geqslant |z|_2 \sqrt{\sigma_k^2 + \sigma_l^2} \times t } \leqslant e^{1-c_* t^2}
\end{align}
This implies that $\forall t \geqslant 0, \proba{ |(W_2)_{ab}| \geqslant |\mu_{k}-\mu_{l}|_2 (2\delta + \sqrt{\sigma_k^2 + \sigma_l^2}\times t) } \leqslant e^{1-c_* t^2}$. We conclude with a union bound over all $(a,b) \in G_k \times G_l$, a union bound over all $(k, l) \in [K]^2$, $k\neq l$ and by taking $t = \sqrt{(1+3\log n)/c_*}$. \qed

\subsubsection{Proof of \lemref{w3}: control of $|\langle W_4, B^*-B\rangle|$}
\label{sec:lemma4}
Recall $(W_4)_{ab} = \langle E_a ,E_b\rangle-{\Gamma}_{ab}$. We will prove \lemref{w3} by using the derivation of \eqref{opinf} combined with \lemref{martsubG} for control of the operator norm and the following lemma for the remaining part.

\begin{lemma}
	For $c'_4 > 0$ absolute constant, with probability greater than $1-1/n$:
	\begin{align}
	|B^*W_4|_\infty \leqslant c'_4 \times (\mathcal{V}^2 \sqrt{\log n} + \sigma^2 \log n)/\sqrt{m}.
	\end{align}
\end{lemma}
\begin{proof}
	Let $(a,b) \in G_k \times G_l$, we rewrite $(B^*W_4)_{ab}$ as the sum of the following two terms:
	\begin{align}
	(B^*W_4)_{ab} &= \frac{u_{b}}{|G_{k}|} \times \1_{k = l} + \langle \widetilde{E}_{k}, E_{b} \rangle \text{ with } \left\{
	\begin{array}{ll}
	u_{b} &:= | E_{b} |_2^2 - \Gamma_{bb}\\
	\widetilde{E}_{k} &:= \frac{1}{|G_{k}|} \sum_{c\in G_{k}, c\neq b} E_{c}
	\end{array} \right.
	\end{align}	
	The bound for $u_{b}$ uses \lemref{quad}: $\forall t \geqslant 0$ $\proba{ || E_b |_2^2 - \E | E_b |_2^2| \geqslant \mathcal{V}_l^2 \sqrt{t} + \sigma_l^2 t } \leqslant 2e^{-c_*t}$ so only the scalar product remains to be controlled. Notice that by \lemref{subg}, $\sqrt{|G_{k}|}\widetilde{E}_{k}$ is a centered subgaussian with variance-bounding matrix $\widetilde{\Sigma} = \frac{1}{|G_{k}|} \sum_{c\in G_{k}, c\neq b} \Sigma_c$, therefore $|\widetilde{\Sigma}|_F \leqslant \mathcal{V}_k^2$ and $|\widetilde{\Sigma}|_{op} \leqslant \sigma_k^2$. So using \lemref{quad} again we find $\forall t \geqslant 0$:
	\begin{align}
	\proba{ 2| \sqrt{|G_{k}|}\langle \widetilde{E}_{k}, E_b \rangle| \geqslant \sqrt{2} \langle \widetilde{\Sigma} , \Sigma_b \rangle^{1/2} \sqrt{t} + |\widetilde{\Sigma}^{1/2} \Sigma_b^{1/2}|_{op} t } \leqslant 2 e^{-c_*t}
	\end{align}
	Therefore using a union bound, then $\langle \widetilde{\Sigma} , \Sigma_b \rangle^{1/2} \leqslant \mathcal{V}_k \mathcal{V}_l \leqslant \mathcal{V}^2$ (Cauchy-Schwarz) and applying another union bound over all $(a,b)\in [n]^2$ with $t = (\log 4 + 3 \log n)/c_*$ yields the result.
\end{proof}

\noindent
We are ready to wrap-up the proof. From \lemref{martsubG} applied to $W_4$, taking $t=(\log 2 + n\log 9 + \log n)/c_*$ there exists $c''_4 > 0$ absolute constant such that we have with probability greater than $1-1/n$: $	|W_4|_{op} \leqslant c''_4 (\mathcal{V}^2 \sqrt{n} + \sigma^2 n)$. Now applying \lemref{deter} to $W_4$:
\begin{align}
|\langle W_4, B^* - B\rangle| \leqslant 6|B^*W_4|_\infty\sum_{1 \leqslant k\neq l \leqslant K} |B_{G_k G_l}|_1 + | W_4 |_{op} \Big[ {\sum_{1 \leqslant k\neq l \leqslant K} |B_{G_k G_l}|_1}/{m} + (\tr(B)-K) \Big]
\end{align}
Therefore combining the lemma with the derivations above and a union bound, we get with probability greater than $1-2/n$:
\begin{align}
|\langle W_4, B^* - B\rangle| \leqslant &\Big[ {6c'_4}(\mathcal{V}^2 \sqrt{\log n} + \sigma^2 \log n)/\sqrt{m} + c''_4(\mathcal{V}^2 \sqrt{n} + \sigma^2 n)/m \Big] \sum_{1 \leqslant k\neq l \leqslant K} |B_{G_k G_l}|_1 \nonumber \\&\qquad+ (\tr(B)-K)_+ c''_4(\mathcal{V}^2 \sqrt{n} + \sigma^2 n) 
\end{align}
This concludes the proof for \lemref{w3}. \qed

\subsection{Proof of Proposition \ref{prop:gamma}, Gamma estimator $\widehat{\Gamma}^{corr}$}
Let $a \in G_k$, $b_1 \in G_{l_1}, b_2 \in G_{l_2}$, using \eqref{latent} and $2|xy|\leqslant x^2 + y^2$ we have for $a\in [n]$:
\begin{align}
\label{gammahat}
|\widehat{\Gamma}_{aa} - \Gamma_{aa}| = |\langle X_{a}-X_{b_{1}}, X_{a}-X_{b_{2}}\rangle - \Gamma_{aa}| &\leqslant U_1 + \frac{3}{2} U_2 + 2U_3 + 3U_4\\
\text{ where: }\quad
U_1 &:= |{ |E_{a} |_2^2} - \Gamma_{aa}| \nonumber\\
U_2 &:= |\nu_a - \nu_{b_1} |_2^2 + |\nu_a - \nu_{b_2} |_2^2 \nonumber\\
U_3 &:= \sup_{(b,c)\in[n]^2} \langle \frac{\nu_a - \nu_c}{|\nu_a - \nu_c |_2}, E_{b}\rangle^2 \nonumber\\
U_4 &:= \sup_{(b,c)\in[n]^2, b \neq c} | \langle E_{b}, E_{c} \rangle | \nonumber
\end{align}

\noindent
\underline{Control of $U_1 = |{ |E_{a} |_2^2} - \Gamma_{aa}|$}: by using the first inequality from  \lemref{quad} with $t = (2\log n + \log 2)/c_*$ there exists $c_1'>0$ such that with probability greater than $1-1/n^2$:
\begin{align}
U_1 \leqslant c_1' \times ( \mathcal{V}_k^2 \sqrt{\log n} + \sigma_k^2 \log n)
\end{align}

\noindent
\underline{Control of $U_3 = \sup_{(b,c)\in[n]^2} \langle \frac{\nu_a - \nu_c}{|\nu_a - \nu_c |_2}, E_{b}\rangle^2$}: write $z = {(\nu_a - \nu_c)}/{|\nu_a - \nu_c |_2}$ and $Y = \Sigma_{b}^{-1/2}E_b \sim$ subg$(I_p)$ and $A = {\Sigma_{b}^{1/2}}^T (z z^T) {\Sigma_{b}^{1/2}}$, so that: $\langle z, E_b \rangle^2 = E_b^T z z^T E_b = Y^T A Y$ . Because $|z|_2 = 1$ and $z z^T$ is symmetric of rank 1 we have $|A|_F = |A|_{op} = \tr(A) \leqslant \sigma^2$ therefore we use \lemref{qg} with $t = (4\log n+\log 2)/c_*$ and then a union bound over all $(b,c) \in [n]^2$ so that with probability greater than $1-1/n^2$:
\begin{align}
U_3 \leqslant c_3' \times \sigma^2 \log n
\end{align}

\noindent
\underline{Control of $U_4 = \sup_{(b,c)\in[n]^2, b \neq c} | \langle E_{b}, E_{c} \rangle |$}: using the fact that $E_{b}$ and $E_{c}$ are independent and the second inequality of \lemref{quad} with $t = (4\log n + \log 2)/c_*$, a union bound over all $(b,c) \in [n]^2$, there exists $c_4'>0$ such that we have with probability greater than $1-1/n^2$:
\begin{align}
U_4 \leqslant c_4' \times (\sigma^2 \log n + \mathcal{V}^2 \sqrt{\log n})
\end{align}

\noindent
\underline{Control of $U_2 = |\nu_a - \nu_{b_1} |_2^2 + |\nu_a - \nu_{b_2} |_2^2$}: here we use the requirement that all groups are of length at least $m \geqslant 3$, there exist $(a_1,a_2) \in G_k \setminus \{a\}$, $(c,d) \in ([n]\setminus \{a,a_1,a_2\})^2$, let $Z=(X_c-X_d)/|X_c-X_d|_2$. For $a_u \in \{a_1,a_2\}$ we have $\langle X_a-X_{a_u}, Z \rangle = \langle \nu_a - \nu_{a_u}, Z\rangle + \langle E_a - E_{a_u}, Z\rangle$. By independence and \lemref{subg}, $\langle E_a - E_{a_u}, Z\rangle$ is subgaussian with variance bounded by $2\sigma^2$. Therefore using the subgaussian tail bounds of \eqref{tail} and a union bound, there exists $c_2'>0$ absolute constant such that with probability over $1-1/n^2$: $V(a,a_1) \vee V(a,a_2) \leqslant 2\delta + c_2' \sigma \sqrt{\log n}$. Hence for $b_u \in \{b_1, b_2\}$ with probability over $1-1/n^2$:
\begin{align}
|\langle X_a - X_{b_u}, X_c-X_d\rangle| \leqslant (2\delta + c_2' \sigma \sqrt{\log n}) |X_c - X_d|_2
\end{align}
Now suppose $l_1 \neq k$, choose $c\in G_k\setminus \{a\}, d\in G_{l_1}\setminus \{b_1\}$. We have $|X_c - X_d|_2 \leqslant |\mu_k-\mu_{l_1}|_2 + 2\delta + |E_c - E_d|_2$. We also have $\langle X_a - X_{b_1}, X_c -X_d \rangle = \langle \nu_a - \nu_{b_1} + E_a - E_{b_1}, \nu_c - \nu_d + E_c - E_d \rangle = \langle \mu_k - \mu_{l_1} + \delta_{ab}  + E_a - E_{b_1},\mu_k-\mu_{l_1} + \delta_{cd} + E_c - E_d \rangle$ for $\delta_{ab} = (\nu_a - \nu_{b_1}) - (\mu_k - \mu_{l_1})$ and $\delta_{cd} = (\nu_c - \nu_d) - (\mu_k - \mu_{l_1})$. Therefore:
\begin{align}
	|\langle X_a - X_{b_1},X_c-X_d\rangle| &\geqslant |\mu_k - \mu_{l_1}|_2^2/2 - 4\delta|\mu_k - \mu_{l_1}|_2 \\
	&- \frac{1}{2} \langle \frac{\mu_k - \mu_{l_1}}{|\mu_k - \mu_{l_1}|_2}, E_c + E_a - E_d - E_{b_1} \rangle^2  - 2\sup_{(b,c,d)\in [n]^3}\langle \frac{\delta_{cd}}{|\delta_{cd}|_2}, E_b \rangle^2 - 4U_4 - 12\delta^2 \nonumber\\
	&\geqslant |\mu_k - \mu_{l_1}|_2^2/2 - 4\delta|\mu_k - \mu_{l_1}|_2 - 8U'_3 - 2U''_3 - 4U_4 - 12\delta^2
\end{align}
where $U'_3 = \sup_{(b,l)\in[n]\times[K]} \langle \frac{\mu_k - \mu_l}{|\mu_k - \mu_l |_2}, E_{b}\rangle^2$, $U''_3 = \sup_{(b,c,d)\in [n]^3}\langle \frac{\delta_{cd}}{|\delta_{cd}|_2}, E_b \rangle^2$.

\noindent
So combining the last derivations:
\begin{align}
|\mu_k - \mu_{l_1}|_2^2/2 - 4\delta|\mu_k - \mu_{l_1}|_2 \leqslant (2\delta + c_2' \sigma \sqrt{\log n}) (|\mu_k-\mu_{l_1}|_2 + 2\delta + |E_c - E_d|_2) \nonumber\\ + 8U_3' + 2U''_3 + 4U_4 + 12\delta^2
\end{align}
Notice that $U_3',U_3''$ can be controlled exactly as $U_3$ was, and simultaneously: for $c''_3>0$ absolute constant, with probability greater than $1-1/n^2$: $8U'_3 + 2U''_3 \leqslant c''_3 \sigma^2 \log n$.

\noindent
We now control $|E_c - E_d|_2$: notice that by \lemref{subg}, $E_c - E_d$ is subg$(\Sigma_c + \Sigma_d)$. We have $\expect{|E_c - E_d|_2^2} \leqslant |\Sigma_c + \Sigma_d|_* \leqslant 2\gamma^2$, $|\Sigma_c + \Sigma_d|_F \leqslant 2\mathcal{V}^2 \leqslant 2\sigma \gamma$ and $|\Sigma_c + \Sigma_d|_{op} \leqslant 2\sigma^2$. Therefore by the first inequality of \lemref{quad} with $t = (4\log n + \log 2)/c_*$ and a union bound over all $(c,d) \in [n]^2$, there exists $c_2''>0$ absolute constant such that we have simultaneously with probability greater than $1-1/n^2$:
\begin{align}
\sup_{(c,d) \in [n]^2}|E_c - E_d |_2 \leqslant c_2'' \sqrt{\gamma^2 + \sigma \gamma \sqrt{\log n} + \sigma^2 \log n} \leqslant c_2'' ({\gamma} + \sigma \sqrt{\log n})
\end{align}

\noindent
Therefore with a union bound, with probability greater than $1-4/n^2$:
\begin{align}
|\mu_k - \mu_{l_1}|_2^2/2 - (c_2'\sigma \sqrt{\log n} + 6\delta) |\mu_k- \mu_{l_1}|_2 \leqslant& (2\delta + c_2' \sigma \sqrt{\log n}) \big(2\delta + ({\gamma} + \sigma \sqrt{\log n}) (c_2''+\frac{c_3''}{c'_2}+\frac{4c_4'}{c'_2})\big) \nonumber \\&+ 12\delta^2
\end{align}	
Hence for $c'_5>0$ absolute constant we have with probability greater than $1-4/n^2$: $|\mu_k - \mu_{l_1}|_2^2 \leqslant c'_5 (\delta + \sigma\sqrt{\log n})(\delta + \sigma\sqrt{\log n} + \gamma)$. The same control can be derived simultaneously for $|\mu_k - \mu_{l_2}|_2^2$ by replacing $d\in G_{l_1}\setminus \{b_1\}$ by $d'\in G_{l_2}\setminus \{b_1, b_2\}$. We conclude that for $c''_5>0$ absolute constant, we have with probability greater than $1-4/n^2$:
\begin{align}
U_2 \leqslant 2|\mu_k - \mu_{l_1}|_2^2 + 2|\mu_k - \mu_{l_2}|_2^2 + 16\delta^2 \leqslant c''_5 (\delta + \sigma\sqrt{\log n})(\delta + \sigma\sqrt{\log n} + \gamma)
\end{align}
Therefore with a union bound over all four terms $U_1, U_2, U_3, U_4$ and $a\in [n]$, for $c_6, c_7>0$ absolute constants we have with probability greater than $1-c_6/n$: $|\widehat{\Gamma}-\Gamma|_{\infty} \leqslant c_7 (\delta + \sigma\sqrt{\log n})(\delta + \sigma\sqrt{\log n} + \gamma)$. This concludes the proof of Proposition \ref{prop:gamma} \qed

\color{black}
\subsection{Proof of Proposition \ref{prop:motivation}}

For this proof we rely heavily on the proof of Theorem \ref{thm:exactrecovery}: let $\widehat{\Gamma} = 0$ so that $W_5 = \Gamma$, notice that $W_3$ and $W_4$ are centered. We take expectation of \eqref{eq:decomposition}, therefore proving $\langle {\Lambda} + {\Gamma},B^* - B \rangle > 0 \text{ for all } B\in \mathcal{C}_K \setminus \{B^*\}$ is equivalent to proving:
\begin{align}
	\langle S_1 + W_1 + \expect{W_2} + \Gamma,B^* - B \rangle > 0 \text{ for all } B\in \mathcal{C}_K \setminus \{B^*\}
\end{align}

\noindent
Notice that for $(a,b)\in G_k \times G_l$, $\expect{(W_2)_{ab}} \leqslant 2\delta |\mu_k - \mu_l|_2$. Using this in combination with other arguments from the proof of Theorem \ref{thm:exactrecovery}, that is using \eqref{S1}, \eqref{W1} and \eqref{W5K}, we have $\forall B \in \mathcal{C}_K$:
\begin{align}
	\langle S_1,B^* - B \rangle &= \sum_{1 \leqslant k\neq l \leqslant K} \frac{1}{2} | \mu_k - \mu_l |_2^2 |B_{G_k G_l} |_1\\
	|\langle W_1, B^* - B\rangle| &\leqslant \sum_{1 \leqslant k\neq l \leqslant K} \delta^2(6 + \frac{\sqrt{n}}{m})|B_{G_k G_l}|_1 \\
	|\langle \expect{W_2}, B^*-B \rangle | &\leqslant \sum_{1 \leqslant k\neq l \leqslant K} 2\delta|\mu_k - \mu_l|_2 |B_{G_k  G_l} |_1\\
	|\langle W_5, B^* - B\rangle| &\leqslant \sum_{1 \leqslant k\neq l \leqslant K} \frac{7|\Gamma|_V }{2m} |B_{G_k G_l}|_1
\end{align}

\noindent
Thus we have:
\begin{align}
	\langle S_1 + W_1 + \expect{W_2} + W_5,B^* - B \rangle \geqslant \sum_{1 \leqslant k\neq l \leqslant K} \Big[ \frac{1}{2}| \mu_k - \mu_l |_2^2 - 2\delta|\mu_k - \mu_l|_2 - \delta^2 (6 + \frac{\sqrt{n}}{m}) - \frac{7|\Gamma|_V }{2m} \Big] |B_{G_k  G_l} |_1
\end{align}

\noindent
Hence we deduce that there exist $c_0$ absolute constant such that if $\rho^2(\mathcal{G},\bmu,\delta) > c_0 (6 + \sqrt{n}/m)$ and $m\Delta^2(\bmu) > 8|\Gamma|_V$, then we have $\argmax_{B \in \mathcal{C}_K} \langle {\Lambda + \Gamma},B \rangle = B^*$. Lastly as $B^*$ is in $\mathcal{C}_K^{\{0,1\}} \subset \mathcal{C}_K$, this concludes the proof. \qed

\subsection{Proof of Proposition \ref{prop:counter}}

Assume $X_1,...,X_n$ is $(\mathcal{G},\bmu,\delta)$-clustered with caracterizing matrix $B^*$ and define the following:
\begin{itemize}
	\item $\delta = 0$ implying maximum discriminating capacity for $\mathcal{G}$ ie $\rho(\mathcal{G},\bmu,\delta) = + \infty$.
	\item Let \[ B^* := \left[{\scriptsize \begin{array}{ccc} 
		 \boxed{\frac{1}{m}} &  & \\
		 & \boxed{\frac{1}{m}} & \\
		 & & \boxed{\frac{1}{m}}
		\end{array}}\right] \in \mathcal{C}_K^{\{0,1\} } \qquad \text{and} \qquad
		B_1 := \left[{\scriptsize \begin{array}{ccc} 
		\boxed{{2}/{m}} &  & \\
		& \boxed{{2}/{m}} & \\
		& & \boxed{\frac{1}{2m}}
		\end{array}}\right] \in \mathcal{C}_K^{\{0,1\} } \] where $\boxed{\frac{1}{m}}$ represents constant square blocks of size $m$ and value $1/m$, and the other values in the matrices are zeros.
	\item $K=3$ and for some $\Delta > 0$, $\mu_1 = (\Delta/\sqrt{2},0,0)^T$ and $\mu_2 = (0,\Delta/\sqrt{2},0)^T$, $\mu_3 = (0,0,\Delta/\sqrt{2})^T$ so that for $(a,b)\in G_k\times G_l$: $\Lambda_{ab} = \langle \mu_k,\mu_l \rangle = \Delta^2/2 \times \1 \{a \overset{\mathcal{G}}{\sim} b\}$. Then $\Delta^2(\bmu) = \Delta^2$ and $\Lambda = (\Delta^2 /2) m B^*$.
	\item For $\gamma_+ > \gamma_- > 0$ let $\Gamma = \diag{ ( \underbrace{\gamma_+, ..., \gamma_+}_m, \underbrace{\gamma_-,...,\gamma_-}_m, \underbrace{\gamma_-,...,\gamma_-}_m)}$
\end{itemize}

\noindent
Then we have the following: $\langle B^*, \Gamma \rangle = \gamma_+ + 2\gamma_-$, $\langle B_1, \Gamma \rangle = 2\gamma_+ + \gamma_-$, $\langle B^*, \Lambda \rangle = \Delta^2/2 \times 3m$, $\langle B_1, \Lambda \rangle = \Delta^2/2 \times 2m$. Thus we have $\langle B^*, \Lambda+\Gamma \rangle < \langle B_1, \Lambda+\Gamma \rangle$ as soon as $m\Delta^2(\bmu) < 2(\gamma_+ - \gamma_-)$. This concludes the proof. \qed

\section{Subgaussian properties and controls}

\begin{lemma}
	\label{lem:subg}
	$\forall a \in [n]$ let $Y_a \sim$ subg$(\Sigma_a)$, independent, $\Sigma_a \in \mathbb{R}^{d\times d}$ then
	\begin{align}
	\label{rotinv}
	&Y = (Y_1^T,...,Y_n^T)^T \sim \text{subg}(\Diag{\Sigma_a}_{a\in [n]}),\\
	\label{subgaf}
	&Z = \sum_{a\in [n]} c_a Y_a \sim \text{subg}(\sum_{a\in [n]} c_a^2 \Sigma_a).
	\end{align}
\end{lemma}
\begin{proof}
	By independence for $z = \{z_1^T,...,z_n^T\}^T \in \mathbb{R}^{nd}, z_a \in \mathbb{R}^{d}$ we have
	\begin{align*}
	&\expect{e^{ z^T (Y-\E Y)}} = \prod_{a=1}^{n} \expect{e^{z_a^T (Y_a- \E Y_a)}} \leqslant \prod_{a=1}^{n} e^{z_a^T \Sigma_a z_a /2} = e^{z^T \Diag{\Sigma_a}_{a\in [n]} z / 2}\\
	&\expect{e^{ z_1^T (Z- \E Z )}} = \prod_{a=1}^{n} \expect{e^{z_1^T c_a (Y_a-\E Y_a)}} \leqslant \prod_{a=1}^{n} e^{ z_1^T c_a^2 \Sigma_a z_1 /2} = e^{z_1^T (\sum_{a\in [n]} c_a^2 \Sigma_a) z_1/2}
	\end{align*}
\end{proof}

\begin{lemma}{Hanson-Wright inequality for subgaussian variables\\}
	\label{lem:qg}
	Let $Y$ be a centered random vector, $Y\sim$ subg$(I_d)$, let $A$ be a matrix of size $d \times d$. There exists $c_* > 0$ such that for any $t \geqslant 0$
	\begin{align}
	\proba{ |Y^T A Y - \expect{Y^T A Y}| \geqslant |A|_F \sqrt{t} + |A|_{op} t} \leqslant 2 e^{-c_* t}.
	\end{align}
\end{lemma}
\begin{proof}
	A variation of the original Hanson-Wright inequality (Theorem 1.1 from \cite{rudelson}), it holds as $\sigma=1$ bounds the subgaussian norm $|Y|_{\Psi_2} := \sup_{x \in \mathcal{S}_{d-1}} \sup_{p\geqslant 1} p^{-1/2}(\E |x^T Y|^p)^{1/p}$, a consequence of Lemma 5.5 from \cite{vershynin}.
\end{proof}

\begin{lemma}{Subgaussian quadratic forms\\}
	\label{lem:quad}
	Let $E,E'$ be centered, independent random vectors, $E \sim$ subg$(\Sigma)$, $E' \sim$ subg$(\Sigma')$, then for $t \geqslant 0$
	\begin{align}
	&\proba{ || E |_2^2 - \E| E |_2^2| \geqslant |\Sigma|_F \sqrt{t} + |\Sigma|_{op} t } \leqslant 2e^{-c_*t}\\
	&\proba{ 2|\langle E, E' \rangle| \geqslant \sqrt{2}\langle \Sigma,\Sigma' \rangle^{1/2}  \sqrt{t} + |\Sigma^{1/2} \Sigma'^{1/2}|_{op} t } \leqslant 2e^{-c_*t}.
	\end{align}
\end{lemma}

\begin{proof}
	For the first inequality, we use \lemref{qg} with $Y = {\Sigma}^{-1/2} {E}$ and $A = \Sigma$. As for the second inequality, by \lemref{subg} we have $Y = (E^T{\Sigma}^{-1/2},E'^T{\Sigma'}^{-1/2T})^T \sim $ subg$(I_{2d})$. Then let us use \lemref{qg} with
	\[ A = \left( \begin{array}{cc}
	0 & {\Sigma^{1/2}} {\Sigma'^{1/2}} \\
	{\Sigma'^{1/2}}^T {\Sigma^{1/2}}^T & 0 \end{array} \right)\]
	Notice that $|A|_F^2 = 2 \langle \Sigma,\Sigma' \rangle$ and $| A |_{op} \leqslant |{\Sigma^{1/2}} {\Sigma'^{1/2}}|_{op}$ so the results follow.
\end{proof}

\noindent
\textbf{Proof of Lemma \ref{lem:martsubG}: concentration of random subgaussian Gram matrices}.

\noindent
Let $W := \bE \bE^T - \E [\bE \bE^T]$. Using the epsilon-net method as in Lemma 4.2 from \cite{rigollet}, let $\mathcal{N}$ be a 1/4-net for $\mathcal{S}_{n-1}$ such that $|\mathcal{N}| \leqslant 9^n$ (see Lemma 5.2 \cite{vershynin}), we have for $u,v \in \mathcal{S}_{n-1}^2: u^T W v \leqslant \max_{x \in \mathcal{N}} x^T W v + \frac{1}{4} \max_{u \in \mathcal{S}_{n-1}} u^T W v \leqslant \max_{x,y \in \mathcal{N}^2} x^T W y + \frac{1}{2} \max_{u,v \in \mathcal{S}^2_{n-1}} u^T W v$ hence
\begin{align}
	\label{net} |W|_{op} &\leqslant 2 \max_{x,y \in \mathcal{N}^2} x^T W y \quad\text{ and }\quad \proba{ |W|_{op} \geqslant t } \leqslant \sum_{x,y \in \mathcal{N}^2} \proba{ x^T W y \geqslant t/2 }
\end{align}
Notice that this rewrites $x^T W y = \sum_{a=1}^{n}\sum_{b=1}^{n} x_a(E_a^T E_b - \Gamma_{ab})y_b = (\sum_{a=1}^{n} E_a^T x_a) (\sum_{b=1}^{n} E_b^T y_b)^T - \E(\sum_{a=1}^{n} E_a^T x_a) (\sum_{b=1}^{n} E_b^T y_b)^T$. For $x,y \in \mathcal{N}^2$, let $x\otimes \Sigma^{1/2} := (x_1\Sigma_1^{1/2}, ..., x_n\Sigma_n^{1/2})^T \in \mathbb{R}^{np\times p}$ and $Y = (E_1^T\Sigma_1^{-1/2},...,E_n^T\Sigma_n^{-1/2})^T \in \mathbb{R}^{np\times 1}$ (by \lemref{subg} we have $Y \sim$ subg$(I_{np})$). We have
\begin{align}
	x^T W y = Y^T (x \otimes \Sigma^{1/2}) (y \otimes \Sigma^{1/2})^T Y - \E [Y^T (x \otimes \Sigma^{1/2}) (y \otimes \Sigma^{1/2})^T Y]
\end{align}
Now define $A := (x \otimes \Sigma^{1/2}) (y \otimes \Sigma^{1/2})^T$: we have $|A|_{op} \leqslant \max_{a\in[n]} |\Sigma_a|_{op}$ because for $z \in \mathbb{R}^{p}$, $|(x \otimes \Sigma^{1/2}) z|_2^2 = \sum_{b=1}^{n}x_b^2 |\Sigma_b^{1/2} z|_2^2 \leqslant \max_{a\in[n]} |\Sigma_a|_{op} |z|_2^2$ .
As for the Frobenius norm, by Cauchy-Schwarz: $|(x \otimes \Sigma^{1/2}) (y \otimes \Sigma^{1/2})^T|_F^2 = \sum_{a=1}^{n}\sum_{b=1}^{n} x_a^2 y_b^2 |\Sigma^{1/2}_a \Sigma^{1/2}_b|_F^2 \leqslant \max_{a\in[n]} |\Sigma_a|_F^2$. Therefore using \lemref{qg} on $Y$ we have $\forall t\geqslant 0: \proba{ |Y^T A Y - \expect{Y^T A Y}| \geqslant \max_{a\in[n]} |\Sigma_a|_F \sqrt{t}+\max_{a\in[n]} |\Sigma_a|_{op} t } \leqslant 2e^{-ct}$. Hence in conjunction with \eqref{net} we conclude the proof.
\qed

\end{alphasection}

\end{document}